\documentclass[12pt,leqno,twoside]{article}
\usepackage{pdfsync}
\usepackage[square,numbers,sort&compress]{natbib} 
\usepackage{mathrsfs,amssymb,amsfonts,amsthm,amsmath,amsbsy, bbm, bm}
\usepackage[hang]{caption} 
\usepackage[usenames,dvipsnames]{xcolor}
\usepackage[bookmarks,bookmarksnumbered,colorlinks=true,pdfstartview=FitV, linkcolor=Red,citecolor=blue!90!green,urlcolor=blue!90!green]{hyperref}

\allowdisplaybreaks

\def \D {\mathcal{D}}
\def \pd {\mathrm{PD}_{\alpha,\theta}}
\def \pdomega {\mathrm{PD}_{\alpha,\theta|\omega}}
\def \simplex {\overline\nabla_{\infty} }
\def \gen {L_{\alpha,\theta}}	
\def \dualgen {A_{\theta}}

\def \E {\mathbb{E}}

\def \P {\mathbb{P}}

\def \Z {\mathbb{Z}}
\def \P {\mathbb{P}}

\def \w {\ell}

\newtheorem{theorem}{Theorem}[section]
\newtheorem{proposition}[theorem]{Proposition} 
\newtheorem{corollary}[theorem]{Corollary}
\newtheorem{lemma}[theorem]{Lemma} 
 
\newtheorem{definition1}[theorem]{Definition} 
 
\newtheorem{remark1}[theorem]{Remark} 

\long\def\symbolfootnote[#1]#2{\begingroup\def\thefootnote{\hspace*{-1mm}\fnsymbol{footnote}}\footnote[#1]{#2}\endgroup}

\title{\bf 
Dual process in the two-parameter Poisson--Dirichlet diffusion
}

\author{\normalsize
\textsc{Robert C. Griffiths}, {\em Monash University}\\
\normalsize\textsc{Matteo Ruggiero}, {\em University of Torino and Collegio Carlo Alberto}\\
\normalsize\textsc{Dario Span\`o}, {\em University of Warwick}\\
\normalsize\textsc{Youzhou Zhou}, {\em Xi'an Jiaotong Liverpool University}}

\date{ \today}

\begin{document}

\maketitle
\thispagestyle{empty}

\begin{center}
\begin{minipage}{.75\textwidth}
\footnotesize\noindent
The two-parameter Poisson--Dirichlet diffusion takes values in the infinite ordered simplex and extends the celebrated infinitely-many-neutral-alleles model, having a two-parameter Poisson--Dirichlet stationary distribution.
Here we identify a dual process for this diffusion and obtain its transition probabilities. The dual is shown to be given by Kingman's coalescent with mutation, conditional on a given configuration of leaves. Interestingly, the dual depends on the additional parameter of the stationary distribution only through the test functions and not through the transition rates. After discussing the sampling probabilities of a two-parameter Poisson--Dirichlet partition drawn conditionally on another partition, we use these notions together with the dual process to derive the transition density of the diffusion. Our derivation provides a new probabilistic proof of this result, leveraging on an extension of Pitman's P\'olya urn scheme, whereby the urn is split after a finite number of steps and two urns are run independently onwards. The proof strategy exemplifies the power of duality and could be exported to other models where a dual is available.
\\[-2mm]


\end{minipage}
\end{center}


\newpage
\section{Introduction}
The  two-parameter Poisson--Dirichlet diffusion (later simply two-parameter diffusion), was introduced by \cite{P2009} and extends the celebrated infinitely-many-neutral-alleles model of \cite{EK1981} (here obtained when $\alpha=0$) to the case of two parameters $\alpha \in [0,1)$ and $\theta >-\alpha$. This diffusion takes values in the closure of the infinite-dimensional ordered simplex, taken with respect to the product topology of $[0,1]^\infty$, namely
\begin{equation}\nonumber
{\overline\nabla}_\infty := \bigg\{x\in [0,1]^{\infty}:\ x_1\geq x_2\geq \cdots \geq 0,\ \sum_{i\ge1} x_i \leq 1\bigg\},
\end{equation} 
and has infinitesimal operator
\begin{equation}\label{gen:0}
\gen  = \frac{1}{2}\sum_{i,j=1}^\infty x_i(\delta_{ij}-x_j)
\frac{\partial^2}{\partial x_i\partial x_j}-\frac{1}{2}\sum_{i=1}^\infty (\theta x_i+\alpha)\frac{\partial}{\partial x_i}.
\end{equation}
The domain of $\gen $ is taken to be the algebra generated by constants and functions
 $\varphi_{k}(x)=\sum_{i=1}^{\infty}x_i^k$ for $k\ge 2$. This is a dense subalgebra of $C(\simplex)$ (cf., e.g., \cite{EK1981}, end of page 435), and \cite{P2009} showed that the closure of $\gen$ in $C(\simplex)$ generates a Feller semigroup.

The two-parameter diffusion is stationary and reversible with respect to the two-parameter Poisson--Dirichlet distribution $\pd$, with $(\alpha,\theta)$ as above.
This distribution was introduced by \cite{PPY92,P95,PY97} as an extension of Kingman's celebrated Poisson--Dirichlet distribution \cite{K75}. It can be viewed as the law of the limit ranked frequencies of types appearing in a sequence $Y_1,Y_2,\ldots$ of observations sampled from Pitman's \cite{P1996} extension of the Blackwell--MacQueen urn P\'olya urn scheme \cite{BM1973}, hereby recalled: for a non-atomic probability measure $P_0$ on an uncountable state space $S$, it is assumed that $\P(Y_1\in A)=P_0(A)$ and for $n\geq 1$, after the first $n$ observations $y^{(n)}:=(y_{1},\ldots,y_{n})$,
\begin{equation}\label{PU scheme}
\P(Y_{n+1}\in A|y^{(n)})=\frac{\theta+\alpha k }{\theta+n}P_{0}(A)
+\frac{1}{\theta+n}\sum_{j=1}^{k}(n_{j}-\alpha)\delta_{y_{j}^{*}}(A),
\end{equation} 
for every Borel set $A$ of $ S$, where $k$ is the number of distinct values $y_{j}^{*}$ in $y^{(n)}$, with $j=1,\ldots,k$, with respective observed multiplicities $n_{j}$. The scheme reduces to the original Blackwell--MacQueen generalized P\'olya urn scheme when $\alpha=0$.
The distribution $\pd$ and its corresponding urn scheme \eqref{PU scheme} have found numerous applications in a variety of fields. We refer the reader to the following monographs, and references therein: 
\cite{B06} for fragmentation and coalescent theory, \cite{P2006} for excursion theory and combinatorics, \cite{LP09} for Bayesian inference, \cite{TJ09} for machine learning, \cite{F2010} for stochastic population dynamics. 
Probably owing to the popularity of the two-parameter Poisson--Dirichlet distribution, since the contribution of \cite{P2009} there has been a continuous and renewed interest in understanding the finer properties of the two-parameter diffusion, together with other closely related dynamical structures. See  \cite{RW09,FS10,E14,CDRS17,F19,FS19,FRSW2020,Fea21b,RW22,Fea23a,Fea23b}. 
Here it is important to recall that the infinitely-many-neutral-alleles model, corresponding to \eqref{gen:0} with $\alpha=0$, 
can be seen as the \emph{unlabeled} version of a Fleming--Viot diffusion, where the \emph{labels} that represent the information on individual values in the underlying evolving population are lost, and only the relative frequencies of types (arranged in decreasing order) are retained. Cf.~\cite{EK93}, Theorem 9.2.1, or \cite{F2010}, Theorem 5.6. A similar correspondence for the two-parameter diffusion is still object of investigation today, and while \cite{FRSW2020,Fea21b,Fea23a,Fea23b} have recently provided advancements, many aspects are still to be fully understood. 

This paper investigates the dual process of the two-parameter diffusion, and aims to shed some further light on the role of the parameter $\alpha$. Informally, two Markov processes $X_{t}$ and $D_{t}$ taking values in two corresponding state spaces $E$ and $F$, are said to be \emph{dual} to each other with respect to a function $h(x,d)$ (that belongs to the domain of both generators) if the identity 
\begin{equation}\label{duality identity}
\E_{x}[h(X_{t},d)]=\E_{d}[h(x,D_{t})]
\end{equation} 
holds for all $x\in E,d\in F$ and $t\ge0$. See \cite{JK2014} for a review. 
In the above identity, the expectation on the left hand side is taken with respect to the law of $X_{t}$, conditional on $X_{0}=x$, while that on the right hand side is taken with respect to the law of $D_{t}$, conditional on $D_{0}=d$. The function $h(x,d)$ satisfying the  identity is referred to as the \emph{duality function} between $X$ and $D$. Duality is an important tool in the theory of stochastic processes \cite{EK86,Dea23}, which has found widespread applications, among other areas, 
in mathematical population genetics
\cite{M99,BEG00,AS05,HW07,EG09,Bea09,Eea10,Fea11,Cea15,Fea21a,CH23}, 
in statistical physics and interacting particle systems \cite{Gea07,Gea09a,Gea09b,O10,HM11,Fea18}, 
and in statistical inference \cite{PR14,PRS16,ALR21,KKK21,ALR23,KKK24}.  
Knowledge of the dual process and of its properties therefore has key implications, which range from the study of the well-posedness of the martingale problem associated to $X_{t}$, to the possibility of conducting sequential inference on $X_{t}$ given a finite computational budget, to the ability of deriving the transition function of $X_{t}$ through limit arguments. The latter will also be the main application of our results.

Here we relate the dual process of \eqref{gen:0} to Kingman's coalescent \cite{K82}. The literature on Kingman's coalescent and its extensions is vast, and here we simply refer the reader to the reviews \cite{B09,GS2010} and references therein. As shown later (cf.~Theorem \ref{main prop}) the dual process to the two-parameter diffusion turns out to coincide with the process that counts, in Kingman's coalescent, the number of non-mutant lineages ancestral to a sample from the current generation, conditionally given the sample's unlabelled allelic partition. In particular, the dual transition rates still depend only on the parameter $\theta$, like for the $\alpha=0$ model, while the dependence on the additional parameter $\alpha$ is through the test (or duality) functions.
This fact may somewhat be surprising. Indeed, it is well-known that a Kingman's coalescent tree with $n$ leaves can be realised by sampling lineages sequentially, starting from the root, according to the Blackwell--MacQueen P\'olya urn scheme, i.e., \eqref{PU scheme} with $\alpha=0$. Given the connection between \eqref{PU scheme} and the stationary distribution of the two-parameter diffusion, it would thus be natural to expect that $\alpha$ plays a role in the law of the dual process. In fact, this parameter has been shown to have a key role in the construction of the diffusion \cite{RW09,CDRS17} or of its labelled counterpart \cite{FRSW2020}. One would then expect $\alpha$ to influence the deletion of groups both through mutation and coalescence. Our results instead  indicate that these two effects of $\alpha$ on the block-counting rates balance each other out perfectly, giving a zero net effect.
A possible explanation is related to the fact that the dual process describes the {\em conditional} genealogy of a sample given its allelic partition. As an effect of exchangeability, conditionally on the frequencies, the parameters only contribute to the law of the coalescent's {\em block-counting process}, keeping track of the total number of surviving lineages, regardless of their types. 
 

The paper is organized as follows. In Section \ref{sec: preliminiaries} we collect some notation together with some necessary preliminary notions. In Section \ref{sec: duality} we fully characterize the dual process of the two-parameter diffusion. In Section \ref{sec: conditional partitions} we discuss the sampling distributions related to drawing a partition from the two-parameter Poisson--Dirichlet model conditional on another, already observed, partition. These distributions, together with the dual process, are then used in Section \ref{sec: transition} where we derive the transition density of the two-parameter diffusion using the identified dual. This transition density was found in \cite{FSWX2011} through analytical tools.
Our derivation uses instead a probabilistic approach, which exploits a \lq\lq split\rq\rq version of Pitman's two-parameter urn scheme \eqref{PU scheme},  namely is a system of two such urns that grow conditionally independently given the same vector of initial observations. The split urns give, in the limit, a bivariate distribution with identical $\pd$ marginal laws, whereby the size of the common initial sample measures the strength of the dependence between the two coordinates. It turns out that the transition density of the two-parameter diffusion can be interpreted as the conditional distribution of one urn composition given the other in such a construction, where the size of the initial sample is {\em random} and its distribution is given by the block-counting process of the dual. 
Our probabilistic derivation of the transition density therefore exemplifies the power of duality and could be reproduced, at least in principle, to identify the transition density for models when this is unknown, but a dual is nonetheless available, e.g., for coupled Wright--Fisher diffusions \cite{Fea21a}.


\section{Preliminaries}\label{sec: preliminiaries}


A vector  $\eta=(\eta_1,\cdots,\eta_d) \in \mathbb{N}^d, d\in\mathbb{Z}_+$, is called an \emph{integer partition} of $n\in\mathbb{Z}_+$ if $\eta_1\geq\cdots\geq\eta_d>0$ and $|\eta|:=\sum_{i=1}^{d}\eta_i=n$. When useful to make explicit, the \emph{length} of $\eta$ will also be denoted as $l(\eta)=d$. Define $\Gamma:=\cup_{n\geq0}\Gamma_n$, where $\Gamma_0=\{\emptyset\}$ and $\Gamma_n$ is the set of all integer partitions of $n$. {It will be convenient to view a partition $\eta$ as an infinite vector obtained by appending an infinite sequence of zero coordinates to the $l(\eta)$-th part of $\eta$.} For $\omega, \eta\in \Gamma$, we say that $\omega\subset \eta$ if and only if $\omega_i\leq \eta_i$ for all $i\geq1$, so that $(\Gamma,\subset)$ becomes a partially ordered set. Throughout the paper, we will have $n=|\eta|$ and $ \w =|\omega|$.
Given $\eta\in \Gamma$ and $x\in \nabla_{\infty}:=\{x\in \simplex: \sum_{i\ge1}x_{i}=1\}$, define now 
\begin{equation}\label{simple P}
{P}_\eta(x) := \sum_{1\le i_1\ne \cdots \ne i_d< \infty}x_{i_1}^{\eta_1}\cdots x_{i_d}^{\eta_d},
\end{equation} 
so that in particular when $\eta$ is given by a singleton we have $P_{(1)}=1$, and for definiteness we also set
${P}_\emptyset(x):=1$ (note however that $\eta=\emptyset$ will not be needed; cf.~Section \ref{sec: duality}). Let also
\begin{equation}\label{vec P}
\vec{P}_\eta(x)
:=\binom{n}{\eta}\frac{1}{a_1(\eta)!\cdots a_n(\eta)!}{P}_\eta(x), \quad \quad 
\binom{n}{\eta}:=\frac{n!}{\eta_1!\cdots \eta_d!},
\end{equation} 
where $a_k(\eta)$ is the number of groups in $\eta$ with size $k$, so that $\sum_{k=1}^na_k(\eta) = d$ and $\sum_{k=1}^nka_k(\eta) = n$. 
Note that for every $x$, $\vec{P}_\eta(x)$ is a symmetric function in $\eta$, and for every $\eta$, it  is a symmetric function in $x$.
Here and throughout we adopt the same convention used in \cite{EK1981,P2009}, whereby  both functions $P_{\eta}(x)$ and $\vec{P}_\eta(x)$ can be extended continuously from $\nabla_{\infty}$ to $\simplex$, since $\nabla_{\infty}$ is dense in $\simplex$. So for example $P_{(1)}$ equals $1$ and not $\sum_{i\ge1}x_{i}$, which is not continuous on $\simplex$. With a slight abuse of notation, we will use the same symbols to denote their continuous extensions onto $\simplex$.

Here $x\in \nabla_{\infty}$ may be regarded as the vector of relative frequencies in an infinite population with potentially infinitely-many types, or of color frequencies for balls in an urn. If we sample with replacement $n$ items from the population or $n$ balls from the urn, given $x$, then $\vec{P}_\eta(x)$ is the probability that the sample features 
$\eta_{1}$ items/balls of one type/color (regardless of which), $\eta_{2}$ items/balls of a different type/color, etc.

These function assume a prominent role in the framework of partition structures. A partition structure is a family of distributions $\{M_n(\eta),\eta\in \Gamma_{n}\}$ for $n\ge1$, that satisfies the consistency condition whereby, for $\omega\in \Gamma_{n-1}$,
\begin{equation}
\label{consistency}
M_{n-1}(\omega)=\sum_{\eta\in\Gamma_n:\eta\supset\omega}\frac{\binom{n-1}{\omega}\chi(\omega,\eta)}{\binom{n}{\eta}}M_{n}(\eta), 
\end{equation}
where $\chi(\omega,\eta)$ equals $a_{\eta_i}(\eta)$ if $\eta_i-\omega_i=1$ and $\eta_j-\omega_j=0$ for $j\ne i$, 
and zero elsewhere. Here $M_{0}(\emptyset)=M_{1}((1))=1$.
A similar consistency property holds for functions $\eta\mapsto \vec{P}_{\eta}(x)$, given $x\in\nabla_\infty$, as well, which are also partition structures.
By Kolmogorov's consistency condition, we can establish a probability measure $\mathcal{M}(\cdot)$ on integer partitions $\Gamma$. We say that a $\Gamma$-valued stochastic process $\{D_n,n\geq 1\}$ has distribution $\mathcal{M}$, if $D_n\in \Gamma_n$ and has marginal distribution $M_n(\eta)$.

By Kingman's Representation Theorem \cite{K78}, any exchangeable partition structure $\{M_n(\eta),n\geq1\}$ admits the representation 
\begin{equation}\label{Kingman representation}
M_n(\eta):=\E_{\mu}[\vec{P}_\eta(X)]=\int_{\simplex}\vec{P}_\eta(x) \mu(dx),
\end{equation} 
where $\mu$ is a probability measure on $\simplex$, called \emph{representing measure}.
Furthermore, if $\{D_n, n\geq 1\}$ is a family of random partitions with distributions $\{M_n,n\geq1\}$, then $Z:=\lim_{n\to\infty}D_n/n$ exists almost surely and has distribution $\mu$.
For ease of later reference, we state without proof the following Lemma, which is an immediate consequence of the above.

\begin{lemma}\label{weak_convergence}
Let $\{M_n,n\geq1\}$ be a partition structure with representing measure $\mu$. Then 
$$
\mu_n(dx)=\sum_{\eta\in\Gamma_n}M_n(\eta)\delta_{\eta/n}(dx)
$$
converges weakly to $\mu$ as $n\to\infty$.
\end{lemma}
%

When partitions are generated by two-parameter Poisson--Dirichlet distributions, \eqref{Kingman representation} holds with $\mu$ given by $\pd$, and $M_n(\eta)$ is the probability of observing a partition $\eta$ of $n$ from a sample of size $n$ drawn from the two-parameter Chinese restaurant process. This is the law of the unlabelled empirical frequencies induced by the generalized P\'olya urn scheme \eqref{PU scheme}, and yields the Ewens--Pitman sampling formula \cite{P95}
\begin{equation}\label{PSF}
M_n(\eta)
=\E_{\alpha,\theta}[\vec{P}_\eta(X)]
=\binom{n}{\eta}\frac{1}{a_1(\eta)!\cdots a_n(\eta)!}\frac{\prod_{l=0}^{d-1}(\theta+ l\alpha)}{\theta_{(n)}}\prod_{i=1}^d(1-\alpha)_{(\eta_i-1)}.
\end{equation} 
where $a_{(n)}=a(a+1)\cdots(a+n-1)$, $a_{(0)}=1$. The expectation of $P_{\eta}$ with respect to $\pd$ yields instead the so-called \emph{exchangeable partition probability function} \cite{P95}
\begin{equation}\nonumber
\E_{\alpha,\theta}[{P}_{\eta}(X)]
=\frac{\prod_{l=0}^{d-1}(\theta+ l\alpha)}{\theta_{(n)}}\prod_{i=1}^d(1-\alpha)_{(\eta_i-1)}.
\end{equation} 
Cf.~also (36) and (42) in \cite{P1996}. Later we will use the shorter notation $\E_{\alpha,\theta}[{P}_{\eta}]$ and $\E_{\mu}[\vec{P}_\eta]$.

\section{The dual process}\label{sec: duality}

In this section we fully characterize a dual process for the two-parameter diffusion induced by symmetric monomial functions. This is shown to be a pure-death continuous-time Markov chain on integer partitions which coincides with the process counting the number of surviving lineages of each type in a Kingman's coalescent tree, going backward in time, conditional on a starting unlabelled partition of \lq\lq leaves\rq\rq. The dual process does not depend on $\alpha$ through its transition rates, but only through the duality functions. In particular, the block-counting process is the same as the known block-counting process dual to any reversible, neutral Wright--Fisher-type diffusion. The dual we describe  in this Section will be used in Section \ref{sec: transition} for deriving the transition density of the two-parameter diffusion, after providing some additional results in Section \ref{sec: conditional partitions} on the sampling distribution $M_n(\eta)$ conditional on a partially observed partition.

The duality identity \eqref{duality identity}, under reasonably general conditions, can be verified through a similar identity involving the corresponding infinitesimal generators acting on appropriate test functions, which is the object of the next two Lemmas, i.e., $Lh(\cdot,d)(x)=Ah(x,\cdot)(d)$ where $L$ is the generator of $X_{t}$ and $A$ that of the dual. See \cite{JK2014}, Proposition 1.2. To this end, note first that the family of functions $\{{P}_{\eta}(x),\eta\in\Gamma,{\min_{i}\eta_{i}>1}\}$, with $P_{\eta}$ as in \eqref{simple P}, are a linear basis for the algebra taken as the domain of $\gen $. Cf.~\cite{P2009}, Section 2. Recall that these functions are intended as their continuous extension from $\nabla_{\infty}$ to $\simplex$, i.e., they are evaluated on $\nabla_{\infty}$ and extended to $\simplex$ by continuity. Hence ${P}_{(1)}(x)=1$, which for example implies that $1={P}_{(1)}(x){P}_{(1)}(x)={P}_{(2)}(x)+{P}_{(1,1)}(x)$, from which one finds that ${P}_{(1,1)}(x):=1-{P}_{(2)}(x)$. More generally, we would like to compute $\gen$ on all $P_{\eta}$, which the following lemma allows. Denote by $e_{i}$ the canonical vector in the $i$-th direction, and let $\eta - e_i$, when $\eta_i=1$, indicate $(\eta_1,\ldots, \eta_{i-1},\eta_{i+1},\ldots,\eta_d)$. In the next two Lemmas, we assume for simplicity of exposition that $\eta$ is unranked, which has the only purpose of avoiding to account for multiplicity constants. The actual value of $P_{\eta}$ is unaffected since it is symmetric in the $\eta$ components.

\begin{lemma}
For $d=l(\eta)>1$, let $\eta_i=1$. Then
\begin{equation}\label{P_recursion:00a}
{P}_{\eta}(x) = {P}_{\eta-e_i}(x) - \sum_{1\leq j\leq d,j\neq i}{P}_{\eta-e_i+e_j}(x), \quad \quad x\in \nabla_{\infty}.
\end{equation}
\end{lemma}
\begin{proof}
By pre-multiplying ${P}_{\eta}(x)$ by ${P}_{(1)}(x)=1$, we find
\begin{equation}\nonumber
\begin{aligned}
{P}_{\eta}(x)
=&\,
\sum_{k=1}^{\infty}x_{k}\sum_{i_1\ne \cdots \ne i_d}x_{i_1}^{\eta_1}\cdots x_{i_d}^{\eta_d}\\
=&\,\sum_{i_1\ne \cdots \ne i_d\ne k}x_{i_1}^{\eta_1}\cdots x_{i_d}^{\eta_d}x_{k}
+\sum_{j=1}^{d}\sum_{i_1\ne \cdots \ne i_d}x_{i_1}^{\eta_1}\cdots x_{i_j}^{\eta_j+1}\cdots x_{i_d}^{\eta_d}\\
=&\,{P}_{(\eta,1)}(x)+\sum_{j=1}^{d}{P}_{\eta+e_{j}}(x).
\end{aligned}
\end{equation} 
The result is now obtained by letting $\eta$ in the claim be $(\eta,1)$ and $i=d+1$.
\end{proof}

Note that ${P}_{\eta-e_i}$ and ${P}_{\eta-e_i+e_j}$  in the above Lemma are well defined since ${P}_{\eta}$ is symmetric in $\eta$.
The next Lemma, key in identifying the dual process, makes use of \eqref{P_recursion:00a} to show how the operator $\gen $ acts on all functions ${P}_\eta$.

\begin{lemma}\label{generator_lemma}
$\gen $ is well defined on all ${P}_\eta$ in the system of equations (\ref{P_recursion:00a}), recursive on $a_1(\eta)$. In particular $\gen 1=0$, and if $n=|\eta|$  and $d=l(\eta)$, 
\begin{equation}
\gen {P}_\eta = 
\frac{1}{2}\sum_{i:\eta_i>1}\eta_i(\eta_i-1-\alpha){P}_{\eta-e_i}+\frac{1}{2}(\theta + (d-1)\alpha)\sum_{i:\eta_i=1}{P}_{\eta-e_i}
 -\frac{1}{2}n(n+\theta-1){P}_\eta.
\label{recursion:0}
\end{equation}
\end{lemma}
\begin{proof}
This was first proved by \cite{P2009}, Proposition 3.1. In the appendix  we provide an independent proof based on \eqref{P_recursion:00a}. Note in particular that $\gen P_{\eta}=0$ when $\eta=(1)$, which uses the extension by continuity from $\nabla_{\infty}$ to $\simplex$ recalled in the introduction.
\end{proof}

Define now 
\begin{equation}\label{duality function}
g_{\eta}(x):=
\frac{{P}_{\eta}(x)}{\E_{\alpha,\theta}[{P}_{\eta}]}
=\frac{\vec{P}_\eta(x)}{\E_{\alpha,\theta}[\vec{P}_{\eta}]},
\end{equation} 
which are going to be our duality functions. Let also
\begin{equation}\label{lambda_n}
\lambda_{n}:=\frac{1}{2}n(\theta+n-1).
\end{equation} 
The following theorem identifies the dual process of the two-parameter diffusion.

\begin{theorem}\label{main prop}
Let $X$ be the diffusion process corresponding to $\gen $, and let $\{\D_{t} \}_{t\geq 0}$ be a continuous-time death process on
$\Gamma$, with transition rates 
\begin{equation}\label{rates in theorem}
\lambda_{|\eta|}p^{\downarrow}(\eta,\omega), \quad \quad 
p^{\downarrow}(\eta,\omega):=\frac{\eta_{i}{a}_{\eta_i}(\eta)}{|\eta|},\quad \quad 
|\eta|>1,
\end{equation} 
when $\omega$ is the descending arrangement of $\eta-e_i$, and zero elsewhere. Then, for every $\eta\in \Gamma$ and $ x\in\overline\nabla_\infty$, we have
\begin{equation}\label{dual:200}
\E\big [g_\eta(X(t))|X_{0}=x]=\E\big [g_{\D_{t} }(x)|\D_{0}=\eta].
\end{equation}
\end{theorem}
\begin{proof}
The proof essentially follows from Lemma \ref{generator_lemma} together with an argument along the lines of that in Section 2 of \cite{BEG00}.
Since $g_{\eta}(x)$ is symmetric with respect to $\eta$, from Lemma \ref{generator_lemma} we have
\begin{align*}
\gen g_\eta (x)=&
\frac{1}{2}\sum_{i:\eta_i>1}\eta_i(\eta_i-1-\alpha)\frac{\E_{\alpha,\theta}[{P}_{\eta-e_i}]}{\E_{\alpha,\theta}[{P}_{\eta}]}g_{\eta-e_i}(x)\\
 &+\frac{1}{2}(\theta + (d-1)\alpha)\sum_{i:\eta_i=1}\frac{\E_{\alpha,\theta}[{P}_{\eta-e_i}]}{\E_{\alpha,\theta}[{P}_{\eta}]}g_{\eta-e_i}
 -\frac{1}{2}n(n+\theta-1)g_{\eta}(x)\\
 =&\frac{1}{2}\sum_{i:\eta_i>1}\eta_i(\eta_i-1-\alpha)\frac{\theta_{(n)}(1-\alpha)_{(\eta_i-2)}}{\theta_{(n-1)}(1-\alpha)_{(\eta_i-1)}}g_{\eta-e_i}(x)\\
 &+\frac{1}{2}(\theta + (d-1)\alpha)\sum_{i:\eta_i=1}\frac{\theta_{(n)}}{\theta_{(n-1)}[\theta+(d-1)\alpha]}g_{\eta-e_i}
 -\frac{1}{2}n(n+\theta-1)g_{\eta}(x)\\
 =&
\lambda_{|\eta|}\sum_{i:\eta_i>1}\frac{\eta_i}{n}g_{\eta-e_i}(x)
+\lambda_{|\eta|}\sum_{i:\eta_i=1}\frac{1}{n}g_{\eta-e_i}
-\lambda_{|\eta|}g_{\eta}(x),
\end{align*}
If we now let 
 \begin{equation*}\label{p down arrow}
p^{\downarrow}(\eta,\omega)=\frac{\binom{|\omega|}{\omega}}{\binom{|\eta|}{\eta}}\chi(\omega,\eta)=\frac{\eta_i\chi(\omega,\eta)}{n}, \quad 
\chi(\omega,\eta)={a}_{\eta_i}(\eta),\quad 
 i=1,\ldots,l(\eta),
\end{equation*} 
when $\omega$ is the descending arrangement of $\eta-e_i$, and zero otherwise, we can write
\begin{equation}\label{duality on g_eta}
\gen g_\eta (x)=\lambda_{|\eta|}\sum_{\omega\in\Gamma_{|\eta|-1}:\ \omega\subset\eta}[g_{\omega}(x)-g_{\eta}(x)]p^{\downarrow}(\eta,\omega).
\end{equation} 
If we now define $\dualgen$ to be the operator on the right hand side of \eqref{duality on g_eta} acting on $g_{\eta}(x)$ as a function of $\eta$, it is plain that $\dualgen$ defines a pure-death process $\D_t$ with rates as in the claim. Note that when $|\eta|=1$ the rate is null, as $g_{\eta}\propto 1$ and $\gen g_{\eta}=0$ by Lemma \ref{generator_lemma}. 
Now the fact that $\gen g_\eta (x)=\dualgen g_\eta (x)$ implies  \eqref{dual:200} follows from Corollary 4.4.13 in \cite{EK86} in light of the boundedness of $g_\eta (x)$ (cf.~also Proposition 1.2 in \cite{JK2014}).
\end{proof}


Here $\D_t$ is the process describing the group sizes in Kingman's coalescent with mutation at time $t$. The transition rates only depend on the  parameter $\theta$ as in the one-parameter model, while the dependence on the second parameter $\alpha$ is only through \eqref{duality function}.

The next proposition identifies the transition function of the dual. To this end, let 
\begin{equation}\label{BC process}
D_t:=|\D_t|
\end{equation} 
be a death process on $\Z_{+}$ that counts the number of groups in $\D_t$. This jumps from $n$ to $n-1$ at rate $\lambda_{n}$ as in \eqref{lambda_n}, and it is well known \cite{G1980, T1984} that its transition probabilities are
\begin{equation}\label{BC process trans probabilities}
d_{nl}^\theta(t) =
\sum_{k=l}^ne^{-\lambda_{k}t}
(-1)^{k-l}
\frac{(2k+\theta-1)(l+\theta)_{(k-1)}}{l!(k-l)!}
\frac{ n_{[k]} }{ (\theta+n)_{(k)} },\quad 1\le l\le n,
\end{equation}
where $a_{[k]}=a(a-1)\cdots(a-k+1)$ for $k\in\mathbb N$, $a_{[0]}=1$, and $d_{n0}^\theta(t)=1-\sum_{l\ge1}d_{nl}^\theta(t)$. Note here that $\lambda_{k}>0$ for $k\ge2$ even when $-1<\theta\le 0$, and so $d_{nl}^\theta(t)>0$ for $2\le l\le n$. 

In the limit when $n\rightarrow \infty$, the last factor on the right in \eqref{BC process trans probabilities} is replaced by 1, yielding 
\begin{equation}\label{BC process trans probabilities infty}
d_{l}^\theta(t) =
\sum_{k\ge l}e^{-\lambda_{k}t}
(-1)^{k-l}
\frac{(2k+\theta-1)(l+\theta)_{(k-1)}}{l!(k-l)!},\quad l\ge1.
\end{equation}
See, e.g., \cite{G2006}. 
This process first appeared in \cite{G1980}, and counts the number of non-mutant edges back in time from the leaves towards the root in a Kingman's coalescent tree, when the mutation rate is $\theta/2\ge0$ along edges of the tree. It is also commonly known as \emph{block-counting process} of Kingman's coalescent with mutation. Moreover, it indexes a mixture expansion for the transition function of the Fleming--Viot process with parent independent mutation, a measure-valued (labelled) version of \eqref{gen:0} when $\alpha=0$. See \cite{EG1993}, Theorem 1.1. In our setting, we show it plays a key role in the transition function expansion of the two-parameter diffusion, a key argument of which is given by the next proposition.

 Define the hypergeometric probabilities
\begin{equation}\label{cal H}
{\cal H}(\omega |  \eta) = \sum_{1 \leq i_1\cdots \leq i_{\ell(\omega)}}\ 
\sum_{(\kappa_{(1)},\ldots, \kappa_{(\ell(\omega))}) = \omega}
\frac 
{
{\eta_{i_1}\choose {\kappa_1}} \cdots {\eta_{i_{\ell(\omega)}
} \choose {\kappa_{i_{\ell(\omega)}}
} }
}
{
{|\eta|\choose |\omega|}
},\quad \quad \omega \subset \eta.
\end{equation}
Here ${\cal H}(\omega |  \eta)$ is the probability of obtaining an unlabelled colour configuration $\omega$ when sampling $|\omega|$ balls without replacement uniformly from an urn containing a colour configuration $\eta$. 
The $|\omega|$ balls may be sampled all at once or one by one.
\begin{proposition}\label{prop: dual transitions}
Let $\D_{t}$ be as in  Theorem \ref{main prop}. Then its transition probabilities  are
\begin{equation}
q^\theta_{\eta\omega}(t):=\P(\D_{t}=\omega|  \D_{0}=\eta)
= {\cal H}(\omega |   \eta)
d^\theta_{|\eta| |\omega| }(t), \quad \quad \omega \subset \eta,
\quad \quad |\eta|>1,
\label{multiple:0}
\end{equation}
and zero otherwise, with
$d^\theta_{nm}(t)$ for $n\geq m\geq 1$ as in \eqref{BC process trans probabilities}. In particular, $\eta=(1)$ is an absorbing state. 
\end{proposition}
\begin{proof}
Let $n=|\eta|$. The process ${\cal D}_t$ in Theorem \ref{main prop} has overall rates $\lambda_{n}$ and the embedded chain has transition probabilities
\begin{equation}\label{embedded:00}
p^{\downarrow}(\eta,\eta^\prime)
= \frac {
{n-1 \choose \eta^\prime}\chi(\eta^\prime,\eta)
}
{
{n\choose \eta}
}
= a_{\eta_i}(\eta)\frac{\eta_i}{n},\quad \quad 
 \eta_i^\prime = \eta_i-1,\ \eta_j^\prime = \eta_j,\ j\ne i.
\end{equation}
It is easy to realize the transition probability (\ref{embedded:00}) by removing uniformly at random one ball after another, without replacement, from an urn.
Let $\eta$ be the initial color frequency of balls in the urn. If one ball is randomly deleted from from the urn, then the probability that the remaining balls have frequency $\eta^\prime$ is 
$a_{\eta_i}(\eta)\frac{\eta_i}{n} = p^{\downarrow}(\eta,\eta^\prime)$. After $n-m$ deletions a sample of size $m$ with colour configuration $\omega$ is obtained with probability
\begin{align*}
&\sum_{\omega=\omega_{n-m}\subset\cdots\subset\omega_1\subset\omega_0=\eta}
p^{\downarrow}(\omega_0,\omega_1)p^{\downarrow}(\omega_1,\omega_2)\cdots p^{\downarrow}(\omega_{n-m-1},\omega_{n-m})\\
&=\frac {
{n-m\choose \omega}
}
{
{n\choose \eta}
}
\sum_{\omega=\omega_{n-m}\subset\cdots\subset\omega_1\subset\omega_0=\eta}
\chi(\omega,\eta)\\
&= \frac {
{n-m\choose \omega}
}
{
{n\choose \eta}
}
\text{dim}(\omega,\eta),
\end{align*}
where
\begin{equation*} 
\text{dim}(\omega,\eta)
:= \sum_{\omega=\omega_{n-m}\subset\cdots\subset\omega_1\subset\omega_0=\lambda}
\chi(\omega,\eta).
\end{equation*}
After $n$ deletions there will be an empty color frequency $\omega=\emptyset$. 
Denote $\text{dim}(\eta)= \text{dim}(\emptyset,\eta)$ and note that  $\text{dim}(\eta)= {n\choose \eta}$ in agreement with 
$1 = {n\choose \eta}^{-1}{0\choose \emptyset}
\text{dim}(\emptyset,\eta)
$.
 Sampling $m$ balls without replacement is equivalent to the above deletion process, and the sampling probability is 
${\cal H}(\omega|  \eta)$ as in (\ref{cal H}). Thus 
\begin{equation}\nonumber
\begin{aligned}
q_{\eta\omega}^{\theta}(t)
=&\,\mathbb{P}_{\eta}(\D_t=\omega)=\mathbb{P}_{\eta}(\D_t=\omega|   D_t=|\omega|)\mathbb{P}_{\eta}(D_{t}=|\omega|)\\
=&\, \mathbb{P}_{\eta}(\D_t=\omega|   D_t=|\omega|)d_{|\eta||\omega|}^{\theta}(t)={\cal H}(\omega |   \eta)
d_{|\eta||\omega|}(t).
\end{aligned}
\end{equation} 
Finally, the fact that $\eta=(1)$ is absorbing follows immediately from Theorem \ref{main prop} through the fact that $\gen {P}_\eta=0$ when $\eta=(1)$ 
in Lemma \ref{generator_lemma}.
\end{proof}

Note that Proposition \ref{prop: dual transitions} accommodates the full range of values for $\theta>-\alpha$, where $0\le \alpha<1$. In fact, even when $-1<\theta<0$, the transition probabilities in \eqref{multiple:0} are positive for any $\eta$ such that $|\eta|>1$, through \eqref{rates in theorem} and specifically the fact that $\lambda_{n}>0$ for all $n>1$; cf.~\eqref{lambda_n}. The fact that formally $\lambda_{1}<0$ when $-1<\theta<0$ is immaterial, as the dual process is absorbed in $\eta=(1)$.

\section{Two-parameter conditional partition structures}\label{sec: conditional partitions}

In order to derive the transition density of the two-parameter diffusion through duality, we first need to explore in some detail the conditional distribution of a partition $\eta$, generated from the two-parameter Poisson--Dirichlet model (e.g., through \eqref{PU scheme}), conditional on having already observed a subset partition $\omega\subset \eta$. This can be done by appealing to the following generalized P\'olya urn scheme.

Let $ \w =|\omega|$, and suppose an urn contains a single white ball and $ \w $ non-white balls, whose colours are denoted by $Y^ \w =(Y_1,\cdots,Y_ \w )$, with $r\le \w$ different colours.
Define $\pi:Y^ \w \to \omega$ to be the function that maps  the sample $Y^ \w $ into the partition $\omega=(\omega_1,\cdots,\omega_r)\in\Gamma_{ \w }$, determined by the equivalence relation on colours, whereby $Y_i,Y_j$ are in the same group if $Y_i=Y_j$. For every group $\omega_{j}$, one ball is assigned mass $1-\alpha$ and the remaining $\omega_j-1$ balls are assigned mass $1$. The single white ball in the urn is assigned mass $\theta+\alpha r$. Balls are then drawn sequentially as follows. If the white ball is drawn, a ball of a new colour is added to the urn and assigned mass $1-\alpha$, and the mass of the white ball is increased by $\alpha$. If a coloured ball is drawn, it is replaced in the urn together with an additional ball of mass $1$ of the same colour. This urn model is an extension of \eqref{PU scheme} and \cite{H1984}. 

When the total number of balls in the urn is $n$, denote the colours of the $n- \w $ new balls by 
$X^{n- \w }=(X_1,\ldots,X_{n- \w })$. Let the colour configuration be denoted by $\pi(X^{n- \w })$ for the new sample and $\pi(X^{n- \w },Y^ \w )$ for the combined sample. We are interested in the distribution of these colour configurations conditional on $Y^ \w $. To this end, denote by $\text{Dir}(\beta_{1},\ldots,\beta_{r})$ a Dirichlet distribution with parameters $(\beta_{1},\ldots,\beta_{r})$ on the $(r-1)$-dimensional simplex, and let a direct sum notation $\oplus$ indicate an accumulation of the points of two point processes. Finally, denote by $\pdomega$ a $\pd$ distribution conditional on having observed the partition $\omega$, and by $\E_{\alpha,\theta|\omega}$ the corresponding expectation.

\begin{proposition}\label{Bayes:250}
Let $\omega \in \Gamma_ \w , \eta \in \Gamma_n$ and $ \gamma \in \Gamma_{n- \w }$. Then $\pdomega$ equals in distribution the descending arrangement of 
\begin{equation}\label{mixture:10}
Z_{ \w ,r}\mathrm{Dir}(\omega_1-\alpha,\ldots, \omega_r-\alpha)
 \oplus 
(1-Z_{ \w ,r})\mathrm{PD}_{\alpha,\theta+\alpha r},
\end{equation} 
where $Z_{ \w ,r}$ has Beta distribution $\mathrm{Beta}( \w -r\alpha,\theta+r\alpha)$,  independent of everything else. Furthermore, 
\begin{align}
\mathbb{P}(\pi(X^{n- \w })=\gamma |  \pi(Y^ \w )=\omega)=&\,\E_{\alpha,\theta|\omega}\big [\vec{P}_{\gamma}\big ],\label{deFinetti}\\
\mathbb{P}(\pi(X^{n- \w },Y^ \w )=\eta |  \pi(Y^ \w )=\omega)=&\,{\cal H}(\omega | \eta)\frac{\E_{\alpha,\theta}\big [\vec{P}_{\eta}\big ]}{\E_{\alpha,\theta}\big [\vec{P}_{\omega}\big ]},\notag
\end{align}
with $\cal H$ is in \eqref{cal H}.
\end{proposition}
\begin{proof}
Equations \eqref{mixture:10} and \eqref{deFinetti}  follow from the exchangeability of draws in the above described urn model, together with Corollary 20 in \cite{P1996}.
Consider now a path from $\omega$ to $\eta$. Let
$\omega^{i}=\pi(X^{i},Y^ \w ), 1\leq i\leq n- \w $, $ \w ^i=|\omega^i|$ and $r^i$ be the number of colours in $\omega^i$. Set $\omega^{0}=\omega$. When $\pi(X^{n- \w },Y^ \w )=\eta$, there is a path $\omega=\omega^0\subset\cdots\subset\omega^{n- \w }=\eta$. Let $\chi_B(\omega^i,\omega^{i+1})=a_{\omega^{i}_k}(\omega^{i})$ if $\omega^{i+1}$ is obtained from $\omega^{i}$ by adding $1$ to an existing $k$-th component or  $\chi_B(\omega^i,\omega^{i+1})=1$ if a new component is added to $\omega^{i+1}$. 
Then
\begin{equation*}
\mathbb{P}(\omega^{i+1} | \omega^i)=\chi_B(\omega^i,\omega^{i+1})\frac{\E_{\alpha,\theta}\big [{P}_{\omega^{i+1}}\big ]}{\E_{\alpha,\theta}\big [{P}_{\omega^i}\big ]}
= \begin{cases}
\displaystyle {\small a_{\omega_k^i}(\omega^{i})}\frac{\omega_k^i - \alpha}{\theta +  \w ^i},\quad &\text{if a coloured ball is drawn,}\\[3mm]
\displaystyle \frac{\theta+\alpha r^i}{\theta+ \w ^i}, &\text{if the white ball is drawn.}
\end{cases}
\label{identity:65}
\end{equation*}
Note that $\chi_B(\omega,\eta)$ and $\chi(\omega,\eta)$ are conjugate, i.e.
$$
\chi_B(\omega,\eta)=\frac{a_1(\omega)!\cdots a_{|\omega|}(\omega)!}{a_1(\eta)!\cdots a_{|\eta|}(\eta)!}\chi(\omega,\eta).
$$
Evaluating the probability of a path from $\omega$ to $\eta$, leads to
\begin{align*}
\mathbb{P}(\eta|\omega)
=&\,\sum_{\omega=\omega^0\subset\cdots\subset\omega^{n- \w }=\eta}\prod_{i=0}^{n- \w -1}\chi_{B}(\omega^{i},\omega^{i+1})\frac{\E_{\alpha,\theta}\big [{P}_{\omega^{i+1}}\big ]}{\E_{\alpha,\theta}\big [{P}_{\omega^{i}}\big ]}\\
=&\,\left(\sum_{\omega=\omega^0\subset\cdots\subset\omega^{n-w}=\eta}\prod_{i=0}^{n- \w -1}\chi_B(\omega^{i},\omega^{i+1})\right)\frac{\E_{\alpha,\theta}\big [{P}_{\eta}\big ]}{\E_{\alpha,\theta}\big [{P}_{\omega}\big ]}\\
=&\,\frac{a_1(\omega)!\cdots a_ \w (\omega)!\dim(\omega,\eta)}{a_1(\eta)!\cdots a_n(\eta)!}\frac{\E_{\alpha,\theta}\big [{P}_{\eta}\big ]}{\E_{\alpha,\theta}\big [{P}_{\omega}\big ]}\\
=&\,\frac{\binom{ \w }{\omega}\dim(\omega,\eta)}{\binom{n}{\eta}}\frac{\binom{n}{\eta}\frac{1}{a_1(\eta)!\cdots a_n(\eta)!}\E_{\alpha,\theta}\big [{P}_{\eta}\big ]}{\binom{ \w }{\omega}\frac{1}{a_1(\omega)!\cdots a_ \w (\omega)!}\E_{\alpha,\theta}\big [{P}_{\omega}\big ]}
={\cal H}(\omega | \eta)\frac{\E_{\alpha,\theta}\big [\vec{P}_{\eta}\big ]}{\E_{\alpha,\theta}\big [\vec{P}_{\omega}\big ]}
\end{align*}
which gives the second claim.
\end{proof}

Note now that when $\gamma$ gives the counts in the new sample and $\eta$ the counts in the combined sample, we have
$
\mathbb{P}(\pi(X^{n- \w })=\gamma |  Y^ \w )=\mathbb{P}(\pi(X^{n- \w },Y^ \w )=\eta |  Y^ \w )
$
and therefore
\begin{equation}\label{conditional partition structure}
M_{n,\omega}(\eta)
:=\E_{\alpha,\theta|\omega}\big [\vec{P}_{\gamma}\big ]
={\cal H}(\omega |  \eta)\frac{\E_{\alpha,\theta}[\vec{P}_{\eta}]}{\E_{\alpha,\theta}[\vec{P}_{\omega}]}.
\end{equation} 
Hence $M_{n,\omega}(\eta)$ is also a partition structure with representing measure $\pdomega$. We conclude the section with a representation of $\pdomega$ alternative to \eqref{mixture:10}.

\begin{lemma}\label{weak_Con}
$\pdomega$ in Proposition \ref{Bayes:250} can be written
\begin{equation}\nonumber
\pdomega(dy)=\frac{\vec{P}_{\omega}(y)}{\E_{\alpha,\theta}\big [\vec{P}_{\omega}\big ]}\pd(dy).
\end{equation}
\end{lemma}
\begin{proof} 
Let $V$ have distribution $\pd$ and $Y^ \w $ be a sample from the urn.
By Lemma \ref{weak_convergence}, 
$
\mathbb{P}(\pi(Y^ \w )=\omega,V\in dy)=\vec{P}_{\omega}(y)\pd(dy).
$
 Since the marginal distribution are $\mathbb{P}(\pi(Y^ \w )=\omega)=\E_{\alpha,\theta}\big [\vec{P}_{\omega}\big ]$, by Bayes' theorem
\begin{equation}\nonumber
\mathbb{P}(V\in dy | \pi(Y^ \w )=\omega)
=\frac{\vec{P}_{\omega}(y)}{\E_{\alpha,\theta}\big [\vec{P}_{\omega}\big ]}\pd(dy).
\end{equation} 
From (\ref{deFinetti}) in Proposition \ref{Bayes:250}, $\mathbb{P}(V\in dy | \pi(Y^ \w )=\omega)=\pdomega(dy)$, leading to the result.
\end{proof}

\section{Derivation of the transition density}\label{sec: transition}

In this section we apply the results of the previous two sections, namely the identified dual process together with conditional two-parameter partition structures, to derive the transition density of the two-parameter  diffusion.
In preparation to this task, recall that the two-parameter diffusion is reversible with stationary distribution $\pd$ \cite{P2009}. The transition probability $P(t,x,dy)$ is absolutely continuous with respect to $\pd$, and its transition density $p(t,x,y)$ was shown in \cite{FSWX2011} to be
\begin{equation}
p(t,x,y) = 1 + \sum_{m=2}^\infty e^{-\lambda_{m}t}q_m(x,y),
\label{RN:00}
\end{equation}
where
\begin{equation}\nonumber
q_m(x,y) = \frac{2m-1+\theta}{m!}\sum_{n=0}^m(-1)^{m-n}{m\choose n}(n+\theta)_{(m-1)}p_n(x,y), \quad \quad m=2,3,\ldots,
\end{equation}
with $p_0(x,y) = 1$ and
\begin{equation*}\label{pnkernel}
p_n(x,y) 
= \sum_{|\eta|=n}
\frac{\vec{P}_\eta(x)\vec{P}_\eta(y)}
{\E_{\alpha,\theta}\big [\vec{P}_\eta\big ]},\ \ \ \ n\geq 1.
\end{equation*} 
The proof of this fact in \cite{FSWX2011} leverages on a spectral representation by expanding on a proof by \cite{E1992} for the the one-parameter case, i.e., the infinitely-many-neutral-alleles model. Previously, \cite{G1979} had obtained (\ref{RN:00}) in the one-parameter case as a limit from a model with finitely-many types (cf.~also Proposition 4.3 in \cite{GS2012}), and \cite{EG1993} the corresponding version for the labelled model. Analogous transition structures also appear in a family of diffusions defined through the Jack graph, see \cite{Z2023} and references therein. 

Here it is useful to emphasize that there are two possible expansions for the transition density of the model at hand: a spectral expansion in terms of reproducing kernel orthogonal polynomials on $\pd$, and  an expansion as mixture of two-parameter Poisson--Dirichlet distributions. 
The equivalence of the two forms was explained in \cite{GS2010,GS2013} for the one parameter case and in \cite{Z2015} (cf.~Theorem 2.1) for the two-parameter case. In particular, \cite{Z2015} showed that (\ref{RN:00}) is the same as
\begin{equation}
p(t,x,y) ={\tilde d_1^{\theta}}(t)+\sum_{n=2}^\infty d_{n}^{\theta}(t)p_n(x,y),\quad \theta>-1\label{density_lod:01},
\end{equation}
where 
\begin{equation}\label{tilde d1}
\tilde d_{1}^{\theta}(t):=1-\sum_{n=2}^{\infty}d_{n}^{\theta}(t).
\end{equation} 
Note that algebraically 
$\tilde d_{1}^{\theta}(t)=d_0^{\theta}(t)+d_1^{\theta}(t)$
with both $d_0^{\theta}(t),d_1^{\theta}(t)$ non-negative when $\theta \geq 0$, but for $-1 < \theta < 0$ these may not be individually non-negative.

Here we concentrate on the expansion \eqref{density_lod:01} and show how this can be derived using the dual process of Theorem \ref{main prop}. 
To this end, we need the following corollary of the same theorem.

\begin{corollary}\label{Corr24}
For $\eta \in \Gamma_n$, we have
\begin{equation}
\E_x\big[{P}_\eta\big(X(t)\big)\big] =\E_{\alpha,\theta}\big [{P}_\eta\big ]\Bigg (
\tilde d_{n1}^{\theta}(t) +
\sum_{ \w =2}^nd^\theta_{n \w }(t)
\sum_{|\omega|= \w ,\omega\subset\eta}{\cal H}(\omega | \eta)
\frac{{P}_\omega(x)}{\E_{\alpha,\theta}\big [{P}_\omega\big ]}\Bigg ),
\label{expansion:0}
\end{equation}
with $\tilde d_{n1}^{\theta}(t)$ defined as in \eqref{tilde d1} in relation to \eqref{BC process trans probabilities}.
\end{corollary}
\begin{proof}
The proof is immediate by expanding the right-hand side of (\ref{dual:200}) and using (\ref{multiple:0}).
\end{proof}

Note that in light of \eqref{vec P}, we can replace ${P}_{\eta}$ and ${P}_{\omega}$ with $\vec{P}_{\eta}$ and $\vec{P}_{\omega}$ in \eqref{expansion:0}. In this equation there are three  sampling probabilities, namely
\begin{align*}
M_n(\eta)=\E_{\alpha,\theta}[\vec{P}_{\eta}],\quad 
M_{n,t}(\eta)=\E_x[\vec{P}_{\eta}(X(t))],\quad 
M_{n,\omega}(\eta)={\cal H}(\omega |  \eta)\frac{\E_{\alpha,\theta}[\vec{P}_{\eta}]}{\E_{\alpha,\theta}[\vec{P}_{\omega}]}.
\label{3 partition structures}
\end{align*} 
Here $M_n(\eta)$ is as in \eqref{PSF} and it is interpreted as the law of the partition $\eta$ taken at stationarity; $M_{n,t}(\eta)$ is the law of $\eta$ relative to the conditional distribution of $X(t)$, given the initial state $X(0)=x$; finally $M_{n,\omega}(\eta)$ is the law of $\eta$ conditional of having observed $\omega \subset \eta$, which was the object of Section \ref{sec: conditional partitions} (cf.~\eqref{conditional partition structure}).
Since all three are partition structures, we can define the respective associated measures
\begin{equation}\label{3 empiricals}
\begin{aligned}
\mu_{n}(dy)=&\,\sum_{|\eta|=n}M_{n}(\eta)\delta_{\frac{\eta}{n}}(dy),\\
\mu_{n,t,x}(dy)=&\,\sum_{|\eta|=n}\E_{x}[\vec{P}_{\eta}(X_t)]\delta_{\frac{\eta}{n}}(dy),\\
\nu_{\omega,n}(dy)=&\,\sum_{\omega\subset\eta}M_{n,\omega}(\eta)\delta_{\eta/n}(dy).
\end{aligned}
\end{equation} 
Lemma \ref{weak_convergence} deals with the convergence of $\mu_n$, while the following, which in fact is an application of the former, with the convergence of $\nu_{n,\omega}$.

\begin{lemma}\label{lemma conditional convergence}
Let $\nu_{n,\omega}$ be as in \eqref{3 empiricals}. Then $\nu_{n,\omega}$ converges weakly to $\pdomega$.
\end{lemma}
\begin{proof}
Recall from Section \ref{sec: conditional partitions} that $\pi(X^{n- \w },Y^ \w )$ denotes the partition of $n$ induced by the joint sample $(X^{n- \w },Y^ \w )$. 
By Kingman's representation theorem (cf.~Section \ref{sec: preliminiaries}), $\pi(X^{n- \w },Y^ \w )\big /n$ converges almost surely to a random variable $Z$ with values in $\simplex$. Because the conditional distribution of $\pi(X^{n- \w },Y^ \w )\big /n$ is 
$$
\nu_{n,\omega}(dy)=\sum_{|\eta|=n,\ \omega\subset \eta} {\cal H}(\omega | \eta)
\frac{\E_{\alpha,\theta}\big [\vec{P}_{\eta}\big ]}{\E_{\alpha,\theta}\big [\vec{P}_{\omega}\big ]}
\delta_{\eta/n}(dy),
$$
we know that $\nu_{n,\omega}(dy)$ converges weakly to the conditional distribution $\mathbb{P}(Z\in dy |  Y^ \w )$. Conditioning on $Y^ \w $, we know $\pi(X^{n- \w })\big / n$ and $\pi(X^{n- \w },Y^ \w )\big /n$ will both converge to $Z$ almost surely. Then $Z$ has distribution $\pdomega $ due to Lemma \ref{weak_convergence} and equation (\ref{deFinetti}).
\end{proof}

The following Proposition uses the above results to obtain the transition function of the diffusion.

\begin{proposition}\label{theorem33}
Given $X(0)=x$, $X(t)$ has distribution
\begin{equation}
P(t,x,dy)=\tilde d_{1}^\theta(t)\pd(dy) +\sum_{ \w =2}^\infty d^\theta_{ \w }(t)\sum_{\omega:|\omega| = \w }
\pdomega (dy)\vec{P}_\omega(x),
\end{equation}
with $\tilde d_{1}^{\theta}(t)$ as in \eqref{tilde d1}.
\end{proposition}
\begin{proof}
From Corollary \ref{Corr24}  we have
\begin{align*}
\mu_{n,t,x}(dy)=&\tilde{d}_{n,1}^{\theta}(t)\mu_{n}(dy)+\sum_{ \w =2}^{n}d^{\theta}_{n \w }(t)\sum_{|\omega|= \w }\vec{P}_{\omega}(x)\nu_{\omega,n}(dy).\label{expansion:000}
\end{align*}
Then for any bounded continuous function $f$ on $\simplex$, we have 
\begin{align*}
\int_{\simplex}f(y)\mu_{n,t,x}(dy)
=\tilde{d}_{n,1}^{\theta}(t)\int_{\simplex}f(y)\mu_{n}(dy)+\sum_{ \w =2}^{n}d^{\theta}_{n \w }(t)\sum_{|\omega|= \w }\vec{P}_{\omega}(x)\int_{\simplex}f(y)\nu_{\omega,n}(dy).
\end{align*}
We also know that $\mu_{n}$ converges weakly to $ \pd$ from Lemma  \ref{weak_convergence} and $\nu_{\omega,n}$ converges weakly to $  \pdomega$ from Lemma \ref{lemma conditional convergence}. Then the bounded convergence theorem implies
\begin{align*}
\lim_{n\to\infty}\int_{\simplex}f(y)\mu_{n,t,x}(dy)
=&\,\tilde{d}_{1}^{\theta}(t)\int_{\simplex}f(y)\pd(dy)\\
&\,+\sum_{ \w =2}^{\infty}d^{\theta}_{ \w }(t)\sum_{|\omega|= \w }\vec{P}_{\omega}(x)\int_{\simplex}f(y)\pdomega(dy).
\end{align*}
Since Lemma  \ref{weak_convergence} also implies that $\mu_{n,t,x}(dy)$ converges weakly to $P(t,x,dy)$, we can conclude that
$$
P(t,x,dy)=\tilde{d}_{1}^{\theta}(t)\pd(dy)+\sum_{ \w =2}^{\infty}d^{\theta}_{ \w }(t)\sum_{|\omega|= \w }\vec{P}_{\omega}(x)\pdomega(dy)
$$
giving the result.
\end{proof}

The transition function obtained in Proposition \ref{theorem33} can be interpreted in terms of P\'olya urn schemes. More specifically, we can exploit the P\'olya urn representation discussed in Section \ref{sec: conditional partitions} to construct, for each $t\geq 0$, a random variable $X(t)$ with distribution $P(t,x,dy)$ as in Proposition \ref{theorem33}. In short, this is obtained by branching the urn scheme after the observation of a sequence that generates a configuration $\omega$, and letting the two resulting split urns evolve independently onwards. This approach is inspired by a similar construction in \cite{GS2012} for the one-parameter case (see also \cite{GJS2018} for a connection with Wright--Fisher diffusion bridges).

Let $\{U_n,n\geq 1\}$ be the process driven by the P\'olya urn scheme described in Section \ref{sec: conditional partitions}, where the urn starts with a single white ball and $U_{n}$ is the configuration of balls after $n$ draws. The urn scheme is run for a random number of draws $D_t$ to obtain a configuration $U_{D_t}$,  where 
$
\mathbb{P}(D_t= \w )=d_{ \w }^{\theta}(t), ~ \w \geq1,
$
and $D_{t}$ is as in \eqref{BC process} and $d_{ \w }^{\theta}(t)$ as in \eqref{BC process trans probabilities infty}.
The urn is then split and two urns are run independently for $n-D_t$ additional draws, for $n>D_{t}$, both beginning at $U_{D_t}$. This produces two random partitions $\eta$ and $\tilde{\eta}$, induced by the two resulting urn configurations. Conditional on the partition $\omega$ obtained at step $D_t$, these partitions are independent and with common distribution \eqref{conditional partition structure}, namely
\begin{equation}\label{eta given omega}
\mathbb{P}(U_n=\eta |  U_{D_t}=\omega)={\cal H}(\omega | \eta)\frac{\E_{\alpha,\theta}\big [\vec{P}_{\eta}\big ]}{\E_{\alpha,\theta}\big [\vec{P}_{\omega}\big ]}.
\end{equation} 
The two urns, denoted $\{(U_n,U^{\prime}_n),n\geq1\}$, are thus coupled, with joint distribution 
$$
\mathbb{P}(U_n=\eta,U^{\prime}_n=\eta^\prime)=\sum_{ \w =1}^{\infty}d_{ \w }^{\theta}(t)\sum_{|\omega|= \w }\E_{\alpha,\theta}\big [\vec{P}_{\omega}\big ]\mathbb{P}(\eta | \omega)\mathbb{P}(\eta^{\prime} | \omega),
$$
where $\mathbb{P}(\eta | \omega)$ is a shorthand notation for \eqref{eta given omega} and we set $\mathbb{P}(\eta | \omega)=0$ if $|\omega| > |\eta|$, and similarly for $\eta^\prime$.
Consider now the measure on $\simplex\times\simplex$
$$
\nu_n(dx,dy)=\sum_{|\eta|=n}\sum_{|\eta^{\prime}|=n}\mathbb{P}(U_n=\eta,U^{\prime}_n=\eta^{\prime})\delta_{\frac{\eta}{n}}(dx)\delta_{\frac{\eta^{\prime}}{n}}(dy).
$$
By another application of Lemma \ref{weak_convergence}, $\nu_{n}(dx,dy)$ converges weakly to
$$
\sum_{ \w =1}^{\infty}d_{ \w }^{\theta}(t)p_ \w (x,y)\pd(dx)\pd(dy).
$$
This probabilistic construction shows that the conditional distribution of $X(t)$ given $  X(0)$, where $X$ is the two-parameter diffusion, is the same as (\ref{density_lod:01}), at a fixed time $t$.

\appendix
\section{Appendix}\label{sec:proof}


\subsection{Proof of Lemma \ref{generator_lemma}}

In this proof $\eta$ and the modifications made to it are not ranked, however it is assumed that singletons are arranged to be at the right end of $\eta$. Let $d=l(\eta)$. The proof that (\ref{recursion:0}) holds is by induction on $l(\eta)$. If $a_1(\eta) = 0$ then $P_\eta \in {\cal C}$ and 
\begin{equation*}
\gen P_\eta = -\frac{1}{2}n(n+\theta-1)P_\eta
 + \frac{1}{2}\sum_{i:\eta_i>1}\eta_i(\eta_i-1-\alpha)P_{\eta-e_i}
\end{equation*}
by simply applying $\gen $. If $a_1(\eta)>0$
 $P_\eta$ can be recursively expressed as linear combinations of functions in ${\cal C}$ which can then be acted on by the differential operator (\ref{gen:0}). Recall that $P_\eta$ is exchangeable in the elements of $\eta$ so it is possible to rearrange the elements so that the $a_1(\eta)$ singletons are the last entries of $\eta$. Suppose that $\eta_d=1$. Then we use the notation that $\eta^-= \eta-e_d = (\eta_1,\ldots,\eta_{d-1})$, 
where the final component of $\eta$ is removed.
 Consider (\ref{P_recursion:00a}) with $i=d$, then
\begin{equation}
P_\eta = P_{\eta^-} - \sum_{j=1}^{d-1}P_{\eta^-+e_j}.
\label{P_recursion:0}
\end{equation}
Suppose that (\ref{recursion:0}) holds for $l(\eta) \leq d-1$. First use the induction hypothesis on the second term on the right of (\ref{P_recursion:0}).
Denote $d^\circ = d - a_1(\eta)$.
For $1\leq j\leq d^\circ$,
\begin{align*}
\gen P_{\eta^-+e_j} 
=&\,-\frac{1}{2}n(n+\theta-1)P_{\eta^-+e_j}\\
&\, + \frac{1}{2}\sum_{i=1}^{d^\circ}(\eta_i+\delta_{ij})(\eta_i+\delta_{ij}-1-\alpha)P_{\eta^--e_i+e_j}
 +\frac{1}{2}(\theta + (d-2)\alpha)(a_1(\eta)-1)P_{\eta^{--}+e_j}\\
=&\,-\frac{1}{2}n(n+\theta-1)P_{\eta^-+e_j}
 + \frac{1}{2}\sum_{i=1}^{d^\circ}\eta_i(\eta_i-1-\alpha)P_{\eta^--e_i+e_j} + \frac{1}{2}(2\eta_j-\alpha)P_{\eta^-}\\
&\,+\frac{1}{2}(\theta + (d-2)\alpha)(a_1(\eta)-1)P_{\eta^{--}+e_j}
\end{align*}
If $a_1(\eta) = 1$ the last term in the equation above with a factor $(a_1(\eta)-1)$ is taken to be zero, and similarly in equations that follow.
Denote $\eta^* = (\eta_1,\ldots,\eta_{d^\circ},2,1,\ldots,1)$ with $a_1(\eta^*)=a_1(\eta)-2$.
For $d^\circ < j \leq d-1$, $P_{\eta^-+e_j}=P_{\eta^*}$ and
\begin{align*}
\gen P_{\eta^-+e_j}
=&\,\gen P_{\eta^*} \\
=&\,-\frac{1}{2}n(n+\theta-1)P_{\eta^*}
 + \frac{1}{2}\sum_{i=1}^{d^\circ}\eta_i(\eta_i-1-\alpha)P_{\eta^*-e_i}
+ \frac{1}{2}2(1-\alpha)P_{\eta^-} \\
&\,
 +\frac{1}{2}(\theta + (d-2)\alpha)(a_1(\eta)-2)P_{\eta^{*-}},
\end{align*}
A term on the right comes from $\eta_{i}=2$, when $i=d^\circ+1$. Then 
$\eta_i(\eta_i-1-\alpha)P_{\eta^-}= 2(1-\alpha)P_{\eta^-}$.
Summing, recalling that $\eta^*=\eta^-+e_j$, and using (\ref{P_recursion:0}) for identities, where $1 \leq i \leq d^\circ$,
\[
\sum_{j=1}^{d-1}P_{\eta^-+e_j} = P_{\eta^-} - P_\eta,\>
\sum_{j=1}^{d-1}P_{\eta^--e_i+e_j} = P_{\eta^--e_i} - P_{\eta-e_i},\>
\sum_{j=1}^{d-2}P_{\eta^{--}+e_j} = P_{\eta^{--}} - P_{\eta^-},
\]
gives
\begin{equation}
\begin{aligned}
\sum_{j=1}^{d-1}\gen P_{\eta^-+e_j} 
=&\,
-\frac{1}{2}n(n+\theta-1)(P_{\eta^-}-\hat P_{\eta})\\
&\,+ \frac{1}{2}\sum_{i=1}^{d^\circ}\eta_i(\eta_i-1-\alpha)(P_{\eta^--e_i}-P_{\eta-e_i})\\
 &\,+ \frac{1}{2}(2(n-a_1(\eta)) - (d-a_1(\eta)))\alpha P_{\eta^-}
 + \frac{1}{2}2(1-\alpha)(a_1(\eta)-1)P_{\eta^-} \\
  &\,+\frac{1}{2}(a_1(\eta)-1)(\theta + (d-2)\alpha)(P_{\eta^{--}} - P_{\eta^-}),
\end{aligned}
\label{proof:011}
\end{equation}
Now use the induction hypothesis on the first term on the right of (\ref{P_recursion:0}).
\begin{align*}
\gen P_{\eta-}
=&\,-\frac{1}{2}(n-1)(n+\theta-2)P_{\eta^-}\\
 &+ \frac{1}{2}\sum_{i=1}^{d^\circ}\eta_i(\eta_i-1-\alpha)P_{\eta^--e_i} +\frac{1}{2}(\theta + (d-2)\alpha)(a_1(\eta)-1)P_{\eta^{--}}\\
 &=-\frac{1}{2}n(n+\theta-1)P_{\eta^-}+ \frac{1}{2}(2n-2+\theta)P_{\eta^-}\\
 &+ \frac{1}{2}\sum_{i=1}^{d^\circ}\eta_i(\eta_i-1-\alpha)P_{\eta^--e_i} +\frac{1}{2}(\theta + (d-2)\alpha)(a_1(\eta)-1)P_{\eta^{--}}
 \end{align*}
 Subtracting (\ref{proof:011}) from the previous yields
\begin{equation*}
\begin{aligned}
\gen \hat P_{\eta} 
=&\,\gen P_{\eta-} - \sum_{j=1}^{d-1}\gen P_{\eta^-+e_j}\\
&= -\frac{1}{2}n(n+\theta-1)\hat P_{\eta} 
 + \frac{1}{2}\sum_{i=1}^{d^\circ}\eta_i(\eta_i-1-\alpha)P_{\eta-e_i} + R(\eta),
\end{aligned}
\end{equation*}
where 
\begin{equation}
\begin{aligned}
R(\eta) =&\, \frac{1}{2}(2(n-1) +\theta)P_{\eta^-}
 -\frac{1}{2}(2(n-a_1(\eta)) - (d-a_1(\eta)))\alpha P_{\eta^-}\\
& - \frac{1}{2}2(1-\alpha)(a_1(\eta)-1)P_{\eta^-} 
  +\frac{1}{2}(a_1(\eta)-1)(\theta + (d-2)\alpha)P_{\eta^-}
\end{aligned}
\label{rcalc:0}
\end{equation}
The coefficient of $\frac{1}{2}a_1(\eta)P_{\eta^-}$ in (\ref{rcalc:0}) is
\[
2-\alpha - 2(1-\alpha) + \theta + (d-2)\alpha = \theta+(d-1)\alpha
\]
and the terms not involving $a_1(\eta)$ are one-half times
\[
2(n-1) + \theta -2n + d\alpha + 2(1-\alpha) - (\theta+(d-2)\alpha) = 0.
\]
Therefore, correctly,  $R(\eta) = \frac{1}{2}( \theta+(d-1)\alpha)P_{\eta^-}$ and the induction is completed.
%



\section*{Ackowledgements}

MR acknowledges support of MUR - Prin 2022 - Grant no.~2022CLTYP4, funded by the European Union - Next Generation EU. YZ was supported by grants NSFC117015170 and RDF-22-01-013.


\begin{thebibliography}{00}

\setlength{\parskip}{-1mm}

\bibitem[]{AS05} {Arthreya, S.} and { Swart, J.} (2005). Branching-coalescing particle systems. \emph{Probab. Theory and Relat. Fields} \textbf{131}, 376--414. 

\bibitem[]{ALR21} {Ascolani, F., Lijoi, A.} and {Ruggiero, M.} (2021). Predictive inference with Fleming--Viot-driven dependent Dirichlet processes. \emph{Bayesian Anal.} \textbf{16}, 371--395.
%
\bibitem[]{ALR23} {Ascolani, F., Lijoi, A.} and {Ruggiero, M.} (2023). Smoothing distributions for conditional Fleming--Viot and Dawson--Watanabe diffusions. \emph{Bernoulli} \textbf{29}, 1410-1434.

\bibitem[]{BEG00} {Barbour, A.D., Ethier, S.N.} and {Griffiths, R.C.} (2000). A transition function expansion for a diffusion model with selection. \emph{Ann. Appl. Probab.} \textbf{10}, 123--162.

\bibitem[]{B09} {Berestycki, N.}  (2009). Recent progress in coalescent theory. \emph{Ensaios matematicos}. Sociedade brasileira de matematica. \textbf{16}, 1--193.

\bibitem[]{B06} {Bertoin, J.} (2006). \emph{Random Fragmentation and Coagulation Processes}. Cambridge University Press, Cambridge.

\bibitem[]{Bea09} {Birkner, M.C., Blath, J., Moehle, M., Steinruecken, M.} and {Tams, J.} (2008). A modified lookdown construction for the Xi-Fleming--Viot process with mutation and populations with recurrent bottlenecks. \emph{Alea} \textbf{6}, 25--61.

\bibitem[]{BM1973}
Blackwell, D. and MacQueen, J.B. (1973). Ferguson distributions via P{\'o}lya urn
schemes. \emph{Ann. Statist.} \textbf{1}, 353--355.

\bibitem[]{Cea15} {Carinci, C., Giardin\`a, C., Giberti, C.} and {Redig, F.} (2015). Dualities in population genetics: A fresh look with new dualities. \emph{Stoch. Proc. Appl.} \textbf{125}, 941--969.

\bibitem[]{CH23} Champagnat, N. and Hass, V. (2023). Existence, uniqueness and ergodicity for the centered Fleming--Viot process. 
\emph{Stoch. Proc. Appl.} \textbf{166}, 103219.

\bibitem[]{CDRS17} 
Costantini, C., De Blasi P., Ethier, S.N., Ruggiero, M. and Span{\`o}, D. (2017). Wright--Fisher construction of the two-parameter Poisson--Dirichlet diffusion. \emph{Ann. Appl. Probab.}  \textbf{27},  1923--1950.

\bibitem[]{DG99} {Dawson, D.A.} and {Greven, A.} (1999). Hierarchically interacting Fleming-Viot processes with selection and mutation: Multiple space time scale analysis and quasi-equilibria. \emph{Electron. J. Probab.} \textbf{4}, 1--81.
%
\bibitem[]{Dea23} {Depperschmidt, A., Greven, A.} and {Pfaffelhuber, P. } (2019). Duality and the well-posedness of a martingale problem. \emph{Theor. Pop. Biol.} \textbf{159}, 59--73.

\bibitem[]{EG09} {Etheridge, A.M.} and { Griffiths, R.C.} (2009). A coalescent dual process in a Moran model with genic selection. \emph{Theor. Pop. Biol.} \textbf{75}, 320--330.

\bibitem[]{Eea10} {Etheridge, A.M., Griffiths, R.C.} and { Taylor, J.E.} (2010). A coalescent dual process in a Moran model with genic selection and the lambda coalescent limit. \emph{Theor. Popn. Biol.} \textbf{78}, 77--92.

\bibitem[]{E1992} Ethier, S. N.(1992). Eigenstructure of the infinitely-many-neutral-alleles diffusion model. \emph{J. Appl. Probab.} {\bf 29}, 487--498.

\bibitem[]{E14}  {Ethier, S. N.} (2014). A property of Petrov's diffusion. \emph{Electron. Comm. Probab.} \textbf{19}, 1--4.

\bibitem[]{EG1993}
Ethier, S.N. and Griffiths, R.C. (1993). The transition function of a Fleming-Viot process. \emph{Ann. Prob.} \textbf{21}, 571--1590.

\bibitem[]{EK1981}
Ethier, S.N and Kurtz, T. G.(1981). The infinitely-many-neutral-alleles diffusion model, \emph{Adv. in Appl. Probab.} \textbf{13},  429--452.
%
\bibitem[]{EK86}  {Ethier, S.N.} and {Kurtz, T.G.} (1986). \emph{Markov processes: characterization and convergence}. Wiley, New York.
%
\bibitem[]{EK93}  {Ethier, S.N.} and {Kurtz, T.G.} (1993). Fleming--Viot processes in population genetics. {\em SIAM J. Control Optim.} {\bf 31}, 345--386.


\bibitem[]{Fea21a} {Favero, M., Hult, H.} and {Koski, T.} (2021). A dual process for the coupled Wright--Fisher diffusion. \emph{J. Math. Biol.} \textbf{82:6}.

\bibitem[]{F2010} Feng, S. (2010) \emph{The Poisson--Dirichlet distribution and related topics.} Springer-Verlag, Berlin, Heidelberg.

\bibitem[]{F19} Feng, S. (2019). Reversible measure-valued processes associated with the Poisson--Dirichlet distribution. \emph{Scientia Sinica Mathematica} \textbf{49}, 377--388.

\bibitem[]{Fea11} Feng, S., Schmuland, B., Vaillancourt, J. and Zhou, X. (2011). Reversibility of interacting Fleming--Viot processes with mutation, selection, and recombination.  \emph{Canad. J. Math.}  \textbf{63}, 104--122.

\bibitem[]{FS10} Feng, S. and Sun, W. (2010). Some diffusion processes associated with two parameter Poisson--Dirichlet distribution and Dirichlet process. \emph{Probab. Theory and Relat. Fields} \textbf{148}, 501--525.


\bibitem[]{FS19} Feng, S. and Sun, W. (2019). A dynamic model for the two-parameter Dirichlet process. \emph{Potential Anal.} \textbf{51}, 147--164.


\bibitem[]{FSWX2011}
Feng, S., Sun, W., Wang, F. Y., and Xu, F. (2011)
 Functional inequalities for the two-parameter extension of the infinitely-many-neutral-alleles diffusion. \emph{J. Funct. Anal.} \textbf{260},  399--413.
 

\bibitem[]{Fea21b} Forman, N., Pal, S., Rizzolo, D. and Winkel, M. (2021). Diffusions on a space of interval partitions: Poisson--Dirichlet stationary distributions. \emph{Ann. Probab.} \textbf{49}, 793--831.
 
 \bibitem[]{Fea23a} Forman, N., Pal, S., Rizzolo, D. and Winkel, M. (2023). Ranked masses in two-parameter Fleming--Viot diffusions. 
\emph{Trans. Amer. Math. Soc.} \textbf{376}, 1089--1111.

\bibitem[]{Fea23b} Forman, N., Rizzolo, D., Shi, Q. and Winkel, M. (2023). Diffusions on a space of interval partitions: the two-parameter model. \emph{Electron. J. Probab.} \textbf{28}, 1--46.


\bibitem[]{FRSW2020} 
Forman, N.,  Rizzolo, D.,  Shi, Q., Winkel, M. (2022)
A two-parameter family of measure-valued diffusions with Poisson--Dirichlet stationary distributions. \emph{Ann. Appl. Probab.} \textbf{32}, 2211--2253.

\bibitem[]{Fea18} {Franceschini, C., Giardin\`a, C.} and {Groenevelt, W.} (2018). Self-duality of Markov processes and intertwining functions. \emph{Math. Phys. Anal. Geom.} \textbf{21}, 29.%


\bibitem[]{Gea07} { Giardin\`a, C., Kurchan, J.} \and {Redig, F.} (2007). Duality and exact correlations for a model of heat conduction. \emph{J. Math. Phys.} \textbf{48}, 033301.

\bibitem[]{Gea09a} { Giardin\`a, C., Kurchan, J., Redig, F.} and {Vafayi, K.} (2009a). Duality and hidden symmetries in interacting particle systems. \emph{J. Stat. Phys.} \textbf{135}, 25--55.
 
 \bibitem[]{Gea09b} { Giardin\`a, C., Redig, F.} and {Vafayi, K.} (2009b). Correlation inequalities for interacting particle systems with duality, \emph{J. Stat. Phys.} \textbf{141}, 242--263.

\bibitem[]{G1979}
Griffiths, R.C. (1979).  A transition density expansion for a multi-allele diffusion model. \emph{Adv. Appl. Probab.} \textbf{11},  310--325. 

\bibitem[]{G1980} Griffiths, R.C. (1980). Lines of descent in the diffusion approximation of neutral Wright-Fisher models.  \emph{Theor. Popul. Biol.} \textbf{17},  37--50.


\bibitem[]{G2006} Griffiths, R.C.  (2006). Coalescent lineage distributions. \emph{Adv. Appl. Probab.} \textbf{38},  405--429.

\bibitem[]{GJS2018} 
Griffiths, R. C., Jenkins, P.A. and Span{\`o}, D. (2018). Wright-Fisher diffusion bridges. \emph{Theor. Popul. Biol.} \textbf{122},  67--77.


%
\bibitem[]{GS2010}
Griffiths, R. C. and Span{\`o}, D. (2010). Diffusion processes and coalescent trees.
Chapter 15 358--375. In: \emph{Probability and Mathematical Genetics, Papers in Honour
of Sir John Kingman.} LMS Lecture Note Series 378, 
ed Bingham, N. H. and Goldie, C. M., Cambridge University Press.

\bibitem[]{GS2012}
Griffiths, R. C., Span{\`o} D.(2011). Multivariate Jacobi and Laguerre polynomials,
infinite-dimensional extensions, and their probabilistic connections with multivariate Hahn and Meixner polynomials. \emph{Bernoulli} \textbf{17}, 1095--1125.


\bibitem[]{GS2013}
Griffiths, R. C., Span{\`o} D.(2013). Orthogonal Polynomial Kernels and Canonical
Correlations for Dirichlet measures. \emph{Bernoulli} \textbf{19},  548--598.


\bibitem[]{H1984} Hoppe, F. (1984). P{\'o}lya like urns and the Ewens' sampling formula. \emph{J. Math. Biol.} \textbf{20}, 91--94.

\bibitem[]{HM11} {Huillet, T.} and {Martinez, S.} (2011). Duality and intertwining for discrete Markov kernels: relations and examples. \emph{Adv. Appl. Probab.} \textbf{43}, 437--460.

\bibitem[]{HW07} {Hutzenthaler, M.} and {Wakolbinger, A.} (2007). Ergodic behavior of locally regulated branching populations. \emph{Ann. Appl. Probab.} \textbf{17}, 474--501.

\bibitem[]{JK2014}
Jansen, S. and Kurt, N. (2014).
On the notion(s) of duality for Markov processes.  \emph{Prob. Surveys} \textbf{11},  59--120.
%


	
\bibitem[]{K75}  {Kingman, J.F.C.} (1975). Random discrete distributions. \emph{J. Roy. Statist. Soc. Ser. B} \textbf{37}, 1--22.

\bibitem[]{K78}
Kingman, J.F.C. (1978). The representation of partition structures. \emph{Journal of the London Mathematical Society} \textbf{2}, 374-380.

\bibitem[]{K82} {Kingman, J.F.C.} (1982). The coalescent. \emph{Stoch. Proc. Appl.} \textbf{13}, 235--248.

\bibitem[]{KKK24} {Kon Kam King, G., Pandolfi, A., Piretto, M.} and {Ruggiero, M.} (2024). Approximate filtering via discrete dual processes. \emph{Stoch. Proc. Appl.} \textbf{168}, 104268.


\bibitem[]{KKK21} {Kon Kam King, G., Papaspiliopoulos, O.} and {Ruggiero, M.} (2021). Exact inference for a class of hidden Markov models on general state spaces. \emph{Electron. J. Stat.} \textbf{15}, 2832--2875.


\bibitem[]{LP09} {Lijoi, A.} and {Pr\"unster, I.} (2009). {Models beyond the Dirichlet process}. In Hjort, N. L., Holmes, C. C., M\"uller, P., Walker, S.G. (Eds.), \emph{Bayesian Nonparametrics}, Cambridge University Press.

\bibitem[]{M99} {M\"ohle, M.} (1999). The concept of duality and applications to Markov processes arising in neutral population genetics
models. \emph{Bernoulli} \textbf{5}, 761--777.

\bibitem[]{O10} {Ohkubo, J.} (2010). Duality in interacting particle systems and boson representation. \emph{J. Stat. Phys.} \textbf{139}, 454--465.

\bibitem[]{PR14} 
Papaspiliopoulos, O. and Ruggiero, M. (2014). Optimal filtering and the dual process. \emph{Bernoulli} \textbf{20},  1999--2019.

\bibitem[]{PRS16} {Papaspiliopoulos, O., Ruggiero, M.} and {Span\`o, D.} (2016). Conjugacy properties of time-evolving Dirichlet and gamma random measures. \emph{Electron. J. Stat.} \textbf{10}, 3452--3489.

\bibitem[]{PPY92} {Perman, M., Pitman, J.} and {Yor, M.} (1992). Size-biased sampling of Poisson point processes and excursions. \emph{Probab. Theory and Relat. Fields} \textbf{92}, 21--39.

\bibitem[]{P2009}
Petrov, L. (2009). Two-parameter family of diffusion processes in the Kingman simplex. \emph{Funct. Anal. Appl.} \textbf{43},  279--296.

\bibitem[]{P95}  {Pitman, J.} (1995). Exchangeable and partially exchangeable random partitions. \emph{Probab. Theory and Relat. Fields} {\bf 102}, 145--158.


\bibitem[]{P1996}  Pitman, J. (1996). Some developments of the Blackwell--MacQueen urn scheme. In \emph{Statistics, Probability and Game Theory} (T.S. Ferguson, L. S. Shapley and J.B. MacQueen, eds.) \emph{IMS Lecture Notes Monogr. Ser.} \textbf{30},  Inst. Math. Statist., Hayward, CA.


\bibitem[]{P2006}
Pitman, J. (2006). \emph{Combinatorial stochastic processes,} volume 1875 of Lecture Notes in Mathematics. Springer-Verlag, Berlin Heidelberg. Lecture notes from {\'E}cole d'{\'e}t{\'e} de Probabilit{\'e}s de Saint-Flour XXXII - 2002  ed. J. Picard.

\bibitem[]{PY97}  {Pitman, J.} and {Yor, M.} (1997). The two-parameter Poisson--Dirichlet distribution derived from a stable subordinator. \emph{Ann. Probab.} \textbf{25}, 855--900.


\bibitem[]{RW22} Rogers, D. and Winkel, M. (2022). A Ray--Knight representation of up-down Chinese restaurants. \emph{Bernoulli} \textbf{28}, 689--712.

\bibitem[]{RW09}  Ruggiero, M., Walker, S.G. (2009). Countable representation for infinite-dimensional diffusions derived from the two-parameter Poisson--Dirichlet process. \emph{Electron. Commun. Probab.} \textbf{14},  501--517.


 \bibitem[]{T1984} 
Tavar{\'e}, S. (1984). Line-of-descent and genealogical processes, and their application in population genetics models. \emph{Theor. Popul. Biol.} \textbf{26}, 119--164.


\bibitem[]{TJ09} {Teh, Y. W.} and {Jordan, M. I.} (2009). {Bayesian nonparametrics in machine learning}. In Hjort, N. L., Holmes, C. C., M\"uller, P., Walker, S. G. (Eds.), \emph{Bayesian Nonparametrics}, Cambridge University Press.

\bibitem[]{Z2015}
Zhou, Y. (2015). Ergodic inequality of a two-parameter infinitely-many-alleles diffusion model.
\emph{J. Appl. Prob.} \textbf{52},  238--246. 

\bibitem[]{Z2023}
Zhou, Y. (2023). Transition density of an infinite-dimensional diffusion with the Jack diffusion. \emph{J. Appl. Probab.} \textbf{60}, 797--811.
\end{thebibliography}
\end{document}